\documentclass[12pt]{amsart}

\usepackage{amsfonts}
\usepackage{amsmath}
\usepackage{amssymb}
\usepackage{amsthm}
\usepackage{graphicx}
\usepackage{color}
\usepackage{epsfig}
\usepackage{url}
\usepackage{mathabx}
\usepackage{placeins}
\usepackage{enumitem}
\usepackage{multicol}

\numberwithin{equation}{section}
\theoremstyle{plain}
\newtheorem{thm}{Theorem}[section]
\newtheorem{rem}{Remark}[section]
\newtheorem{prop}{Proposition}[section]

\newtheorem{cor}{Corollary}[section]
\newtheorem{lem}{Lemma}[section]

\newcommand{\dE}{\mathbb{E}}
\newcommand{\dR}{\mathbb{R}}

\newcommand{\dP}{\mathbb{P}}

\newcommand{\cO}{\mathcal{O}}
\newcommand{\cL}{\mathcal{L}}
\newcommand{\cN}{\mathcal{N}}

\newcommand{\cM}{\mathcal{M}}

\newcommand{\wh}{\widehat}

\newcommand{\wt}{\widetilde}
\newcommand{\wc}{\widecheck}

\newcommand{\indicatrice}{\mathchoice{\rm 1\mskip-4mu l}{\rm 1\mskip-4mu l}{\rm 1\mskip-4.5mu l} {\rm 1\mskip-5mu l}}

\makeatletter    

\newcommand{\unnumberedcaption}%
	{\@dblarg{\@unnumberedcaption\@captype}}

\newcommand{\@unnumberedcaption}{}
\long\def\@unnumberedcaption#1[#2]#3{\par
  \addcontentsline{\csname ext@#1\endcsname}{#1}{%
    \protect\numberline{}{\ignorespaces #2}%
    }%
  \begingroup
    \@parboxrestore
    \normalsize
    \@makeunnumberedcaption{\ignorespaces #3}\par
  \endgroup}

\newcommand{\@makeunnumberedcaption}[1]{%
  \vskip\abovecaptionskip
  \sbox\@tempboxa{#1}%
  \ifdim \wd\@tempboxa >\hsize
    #1\par
  \else
    \global \@minipagefalse
    \hbox to\hsize{\hfil\box\@tempboxa\hfil}%
  \fi
  \vskip\belowcaptionskip}

\@ifundefined{abovecaptionskip}{%
  \newlength{\abovecaptionskip}%
  \setlength{\abovecaptionskip}{10pt}%
}{}
\@ifundefined{belowcaptionskip}{%
  \newlength{\belowcaptionskip}%
  \setlength{\belowcaptionskip}{0pt}%
}{}

\makeatother    

\begin{document}
\title[WLSE for the subcritical Heston process]
{Weighted least squares estimation for the subcritical Heston process \vspace{1ex}}
\author{Marie du ROY de CHAUMARAY}
\address{Universit\'e de Bordeaux, Institut de Math\'{e}matiques de Bordeaux,
UMR 5251, 351 Cours de la Lib\'{e}ration, 33405 Talence cedex, France.}

\begin{abstract}
We simultaneously estimate the four parameters of a subcritical Heston process. We do not restrict ourself to the case where the stochastic volatility process never reaches zero. In order to avoid the use of unmanageable stopping times and natural but intractable estimator, we propose to make use of a weighted least squares estimator. We establish strong consistency and asymptotic normality for this estimator. Numerical simulations are also provided, illustrating the good performances of our estimation procedure.
\end{abstract}

\maketitle

\section{Introduction}
Introduced in 1973, as an hedging tool, the Black-Scholes model uses a geometric Brownian motion to represent asset prices. The implied volatility is supposed to be constant over time, which turned out to be inaccurate to fit real market data, especially during the crash in 1987, see \cite{Stein}. Several alternative models have been constructed to take into account the so-called smile effect associated to deep in-the-money or out-of-money options. A particular attention has been drawn to the study of stochastic volatility processes in which the volatility is also given as a solution of some stochastic differential equation, see \cite{SteinStein}, \cite{Lewis} and \cite{Gath} for financial accuracy. Among them, Heston process \cite{hest} is one of the most popular, due to its computational tractability. For example, \cite{Lee} easily computes call option prices using Fourier inversion techniques. Numerous results about the asymptotic volatility smile can be found in the very recent literature: see e.g. \cite{J1}, \cite{J2}, \cite{J3}.
 
 We denote by $Y_t$ the logarithm of the price of a given asset and by $X_t$ its instantaneous variance, and we consider the following Heston process
\begin{equation} \label{heston}
\left\lbrace 
\begin{array}{c @{\, = \, } l}
    \mathrm{d}X_t & (a+b X_t) \, \mathrm{d}t + 2 \sqrt{X_t}\, \mathrm{d}B_t  \\
    \mathrm{d}Y_t & (\alpha+\beta X_t) \, \mathrm{d}t + 2 \sqrt{X_t}\, \left( \rho \,\mathrm{d}B_t + \sqrt{1-\rho^2} \, \mathrm{d}W_t \right) \\
\end{array}
\right.
\end{equation}
with $a >0$, $\left(b, \alpha, \beta \right) \in \dR^3$ and $\rho \in ]-1,1[$, where $\left(B_t,W_t\right)$ is a 2-dimensional standard Wiener process and the initial state $\left(x_0,y_0\right) \in \dR^{+} \times \dR$.
In this process, the volatility $X_t$ is driven by a generalized squared radial Ornstein-Uhlenbeck process, also known as the CIR process, firstly studied by Feller \cite{fel} and introduced in a financial context by Cox, Ingersoll and Ross \cite{CIR} to compute short-term interest rates.
The asymptotic behavior of this process has been widely investigated and depends on the values of both coefficients $a$ and $b$. 

Once a model has been chosen for its realistic features, it needs to be calibrated before being used for pricing. Our goal in this paper is to estimate parameters $(a,b,\alpha,\beta)$ at the same time using a trajectory of $(X_t)$ and $(Y_t)$ over the time interval $[0,T]$. Azencott and Gadhyan \cite{az} developed an algorithm to estimate some parameters of the Heston process based on discrete time observations, by making use of Euler and Milstein discretization schemes for the maximum likelihood. However, in the special case of an Heston process, the exact likelihood can be computed. It allows us to construct the maximum likelihood estimator (MLE) without using sophisticated approximation methods, which is necessary for many stochastic volatility models, see \cite{AitK}. The MLE of $(a,b,\alpha,\beta)$  has been recently investigated in \cite{pap2}, together with its asymptotic behavior in the special case where $a \geq 2$. 
Denote by $\tau_0$ the stopping-time given by 
\begin{equation}\label{tau0}
\tau_0=\inf \left\lbrace T > 0 \left| \, \int_{0}^T{X_t^{-1} \, \mathrm{d}t}= \infty \right. \right\rbrace.
\end{equation}
For any $a>0$, the MLE $\wt{\theta}_T=\left(\wt{a}_T, \wt{b}_T, \wt{\alpha}_T,
\wt{\beta}_T\right)$ is given, for $T < \tau_0$, by:
\begin{equation}\label{struc}
\wt{\theta}_T=
\begin{pmatrix}
G_T^{-1} & 0\\
0 & G_T^{-1}
\end{pmatrix} \begin{pmatrix}
\wt{U}_T\\
\wt{V}_T
\end{pmatrix}
\end{equation}
where $\wt{U}_T= \left(\int_0^T{X_t^{-1} \, \mathrm{d}X_t},
\int_0^T{X_t \, \mathrm{d}X_t}\right)^{\intercal}$, $\wt{V}_T=\left(\int_0^T{X_t^{-1} \, \mathrm{d}Y_t},
\int_0^T{X_t \, \mathrm{d}Y_t}\right)^{\intercal}$ and
\begin{equation*}
G_T= \begin{pmatrix}
\int_0^T{X_t^{-1} \, \mathrm{d}t} & T \\
T & \int_0^T{X_t \, \mathrm{d}t}
\end{pmatrix}.
\end{equation*}
One can observe that $(\wt{a}_T,\wt{b}_T)$ coincides with the MLE of the parameters $(a,b)$ of the CIR process based on the observation of $(X_T)$ over the time interval $\left[0,T\right]$.
The asymptotic behavior of this latter estimator is well-known, see for example \cite{FT}, \cite{Ov} and \cite{KAB2}. In the supercritical case $b>0$, Overbeck \cite{Ov} has shown that $\wt{b}_T$ converges a.s. to $b$ whereas there exists no consistent estimator for $a$. Hence, we will focus our attention on the geometrically ergodic case $b<0$. Furthermore, the value of $a$ governs the behavior at zero of $(X_T)$: for $a \geq2$, the process almost surely never reaches zero , whereas for $0<a<2$, zero is quite frequently visited and \begin{equation}\label{p0}
\dP\left(\tau_0< \infty\right)=1,
\end{equation}
 see for instance \cite{LL} or \cite{Ov}. For $a>2$, the MLE  converges a.s. to $\theta=\left(a,b,\alpha,\beta\right)$ and satisfies the following Central Limit Theorem (CLT)
$$\sqrt{T}\begin{pmatrix}\wt{\theta}_T-\theta \end{pmatrix} \xrightarrow{\mathcal{L}} \mathcal{N}(0,4D^{-1}) $$
where the block matrix $D$ is given by
$$ D= \begin{pmatrix}
\Sigma & \rho \Sigma\\
\rho \Sigma & \Sigma
\end{pmatrix} \hspace{1cm} \text{with} \hspace{1cm}  \Sigma= \begin{pmatrix}
 \frac{-b}{a-2} & 1 \\
 1 & -\frac{a}{b}
 \end{pmatrix}.$$
 A large deviation principle for the couple $(\wt{a}_T,\wt{b}_T)$ was recently established in \cite{dRdC1}. In the particular case where one parameter is known and the other one is estimated, large deviations can be found in \cite{Z}, while moderate deviations are given in \cite{GJ2}. 
 
By contrast, in the case where $0<a<2$, \eqref{p0} implies the non-integrability of $X_T^{-1}$ for large values of $T$ so that the MLE does not converge for $T$ going to infinity. Consequently, this case has been less investigated even though it is often of interest in finance, to compute long dated interest rates for instance, as explained in \cite{And}, or in FX-markets, see \cite{jan}.  
   In the case of the CIR process, Overbeck \cite{Ov} used accurate stopping times to build a strongly consistent estimator based on the MLE: 
   \begin{equation}
   \indicatrice_{T < \tau_0} \begin{pmatrix}
   \wt{a}_T \\
   \wt{b}_T
   \end{pmatrix} + \indicatrice_{\tau_0 \leq T} \begin{pmatrix}
   \underset{t \to \tau_0}{\lim}{S_t \Sigma_t^{-1}} \\
   \left(\int_0^T{X_s \, \mathrm{d}s}\right)^{-1} \left(X_T-T\underset{t \to \tau_0}{\lim}{S_t \Sigma_t^{-1}} \right)
   \end{pmatrix}
   \end{equation}
   where $S_t=\int_0^t{X_s^{-1} \, \mathrm{d}X_s}$, $\Sigma_t=\int_0^t{X_s^{-1} \, \mathrm{d}s}$ and $\tau_0$ is given by \eqref{p0}. 
    The aim of this paper is to investigate a new strongly consistent weighted least squares estimator (WLSE) for the quadruplet of parameters $\theta$ (and for $\left(a,b\right)$ as a consequence). The weighting allows us to circumvent the explosion for $X_T$ reaching zero and consequently avoid us to make use of stopping times, which are not easy to handle in practice. It generalizes to continuous time the original work of Wei and Winnicki \cite{Wei} for branching processes with immigration, inspired by an analogy with first order autoregressive processes. Our results answer, by the way, the question of Ben Alaya and Kebaier in the conclusion of \cite{KAB2} regarding the CIR process.

 Following the seminal work of \cite{Wei}, denote $C_T=X_T+c$ where $c$ is some positive constant. Our new couple of weighted least squares estimator is given by
\begin{equation}\label{Est}
\wh{\theta}_T = \begin{pmatrix}
\Gamma_T^{-1} & 0\\
0 & \Gamma_T^{-1}
\end{pmatrix} \begin{pmatrix}
U_T\\
V_T
\end{pmatrix}
\end{equation} 
 where $U_T=\left(\int_0^T{\frac{1}{C_t} \, \mathrm{d}X_t},
\int_0^T{\frac{X_t}{C_t} \, \mathrm{d}X_t}\right)^{\intercal}$, $V_T=\left(\int_0^T{\frac{1}{C_t} \, \mathrm{d}Y_t},
\int_0^T{\frac{X_t}{C_t} \, \mathrm{d}Y_t}\right)^{\intercal}$
and
\begin{equation*}
\Gamma_T= \begin{pmatrix}
\int_0^T{\frac{1}{C_t} \, \mathrm{d}t} & \int_0^T{\frac{X_t}{C_t} \, \mathrm{d}t}\\
\int_0^T{\frac{X_t}{C_t} \, \mathrm{d}t} & \int_0^T{\frac{X_t^2}{C_t} \, \mathrm{d}t}
\end{pmatrix}.
\end{equation*} 
We do not restrict ourselves to the case where $c=1$ as it may lower sometimes the variance of the estimators. 
In the particular case where $c=0$, one can observe that the new estimator  coincides with the MLE. 

The paper is organized as follows. The second section contains our main results: the strong consistency of this new couple of estimators as well as its asymptotic normality. The third section deals with a comparison with the MLE,
while the remaining of the paper is devoted to the proofs of our main results, as well as their illustration by some numerical simulations.

\section{Main results}
Our main results are as follows.
\begin{thm}\label{T1}
Assume that $a>0$ and $b<0$. Then, the four-dimensional WLSE $\wh{\theta}_T$ is strongly consistent: for $T$ going to infinity,
\begin{equation}\label{cvps_hest} 
\wh{\theta}_T \xrightarrow{a.s.} \theta.
\end{equation}
\end{thm}
For $T$ going to infinity, $X_T$ converges in distribution to a random variable $X$ with Gamma  $\Gamma(a/2,-b/2)$ distribution, see Lemma 3 of \cite{Ov} for instance. Additionally, we denote by $C$ the limiting distribution of $X_T+c$, as $T$ goes to infinity.

\begin{thm}\label{T2}
Assume that $a>0$ and $b<0$. Then, for $T$ going to infinity, the estimator $\wh{\theta}_T$ satisfies the following CLT 
\begin{equation}\label{TCL_hest}
\displaystyle \sqrt{T} \left(\wh{\theta}_T-\theta\right) \xrightarrow{\cL}\cN\left(0,4\Lambda\right),
\end{equation}
where the asymptotic variance $\Lambda$ is defined as a block matrix by
\begin{equation}
\Lambda=\begin{pmatrix}
ALA & \rho ALA \\
\rho ALA & ALA
\end{pmatrix},
\end{equation}
with the matrix $A$ and $L$ respectively given by \begin{equation*}
A=\left(\dE \left[ C\right]\dE \left[ 1/C\right]-1\right)^{-1}\begin{pmatrix}
 \dE \left[ X^2/C\right] & -\dE \left[ X/C\right] \\
 -\dE \left[ X/C\right] & \dE \left[ 1/C\right]
\end{pmatrix}
\end{equation*} 
and \begin{equation*}
L=\begin{pmatrix}
\dE \left[ X/C^2\right] & \dE \left[ X^2/C^2\right] \\
\dE \left[ X^2/C^2\right] & \dE \left[ X^3/C^2\right]
\end{pmatrix}.
\end{equation*}
\end{thm}
We deduce from the previous theorems the following result for the MLE of the two parameters of the CIR process $\left(X_T\right)$.
\begin{cor}
Assuming that $a>0$ and $b<0$, the WLSE $(\wh{a}_T,\wh{b}_T)$ of parameters $(a,b)$ is strongly consistent for $T$ going to infinity, \begin{equation*}\label{cvps} 
\begin{pmatrix}
\wh{a}_T\\
\wh{b}_T
\end{pmatrix} \xrightarrow{a.s.} \begin{pmatrix}
a\\
b
\end{pmatrix}.
\end{equation*}
and satisfies the following CLT \begin{equation*}\label{TCL}
\displaystyle \sqrt{T} \begin{pmatrix}
\wh{a}_T-a\\
\wh{b}_T-b
\end{pmatrix} \xrightarrow{\cL}\cN\left(0,4A L A\right),
\end{equation*}
\end{cor}

\begin{rem}
In the remaining of this paper, we denote
\begin{equation}\label{defpsi}
\psi_c=\left(-\frac{bc}{2}\right)^{a/2} e^{-bc/2} \, \Gamma\left(1-a/2,-bc/2\right),
\end{equation}
where $\Gamma$ is the upper incomplete gamma function defined for all $y \in \dR$ and $\alpha \in\dR^{+}_{*}$ by
$$\Gamma(\alpha,y)= \int_y^{+\infty} e^{-t} t^{\alpha-1} \, \mathrm{d}t,$$
and extended, for $y \neq 0$, to any real $\alpha$ by holomorphy.
To simplify following expressions, we also define
\begin{equation}\label{defphi}
\varphi_c=\psi_c \left(1-\frac{a}{bc} \right) -1.
\end{equation}
 In the proof of Theorem ~\ref{T2}, we evaluate the two matrices $A$ and $L$ involving $c$ and we obtain that
\begin{equation}A=\varphi_c^{-1}\begin{pmatrix}
c \left(\psi_c-1 \right) -\frac{a}{b} & \psi_c-1 \\
\psi_c-1 & \frac{\psi_c}{c}\\
\end{pmatrix}
\end{equation} and
\begin{equation}L = \frac{1}{2} \begin{pmatrix}
\displaystyle \frac{a}{c} \, \psi_c + b \left(1-\psi_c\right)& \displaystyle \left(a+2-bc\right) \, \left(1-\psi_c\right) - a\\
\displaystyle \left(a+2-bc\right) \, \left(1-\psi_c\right) - a & \displaystyle \psi_c c\, \left(a+4-b\right) -4c-bc^2-\frac{2a}{b}
\end{pmatrix}.
\end{equation}
By a straightforward computation, we deduce that
$A L A=  \left(\varphi_c\right)^{-2}\begin{pmatrix} \sigma_{11} & \sigma_{12} \\
\sigma_{12} & \sigma_{22}
\end{pmatrix}$ where the variances $\sigma_{11}$ and $\sigma_{22}$ are respectively given by $$\sigma_{11}= \left(\psi_c-1\right)^2 \left(\frac{a}{b}-bc^2+\psi_c \frac{bc}{2}\left(c-1\right)\right)-\frac{a^2}{2b}\varphi_c$$ $$\sigma_{22}= \left(\left(\psi_c^2-1\right) \frac{b}{2} +\frac{\psi_c}{c}\right) \varphi_c+ \frac{\psi_c}{2c}\left(\psi_c^2\left(a-b\right)+\psi_c \left(2-bc\right)-2\right),$$ and the covariance $\sigma_{12}$ is given by 
$\sigma_{12}= \left(\psi_c-1\right)^2-\frac{a}{2}\varphi_c$.
\end{rem}
  
\begin{rem}
For $c$ going to zero (for which we need $a$ to be greater than $2$) , we obtain the same covariance matrix than for the MLE. Indeed, using well-known asymptotic results about the incomplete Gamma function $\Gamma$, which could be found in \cite{Luke2}, we have that, as soon as $a>2$,
\begin{equation}
\Gamma\left(1-a/2,-bc/2\right) \underset{c \to 0}{\longrightarrow} 0
\end{equation}
and
\begin{equation}
\Gamma\left(1-a/2,-bc/2\right) \left(-\frac{bc}{2}\right)^{a/2-1} \underset{c \to 0}{\longrightarrow} \frac{1}{1-a/2} = \frac{2}{a-2}.
\end{equation}
Thus $\psi_c$ goes to zero for $c$ tending to zero and $\frac{\psi_c}{c}$ converges to $\frac{-b}{a-2}$. Hence, we easily obtain that, for $c$ going to zero,
\begin{equation*}
A L A \underset{c \to 0}{\longrightarrow} \Sigma^{-1} \hspace{1cm} \text{ where }  \hspace{1cm} \Sigma= \begin{pmatrix}
 \frac{-b}{a-2} & 1 \\
 1 & -\frac{a}{b}
 \end{pmatrix}.
\end{equation*}

\end{rem}
 
\section{Asymptotic variance}
Even though we considered the weighted least squares estimators in order to investigate the case $0<a< 2$ for which the MLE is not consistent, it is interesting to compare the asymptotic variances in the CLT of this new estimators and of the MLE, in the case where $a>2$. This comparison requires a lot of technical calculation as the asymptotic variances depends on the value of $a$, $b$ and $c$. However, it is quite easy to compare variances for the MLE of the parameters of the CIR process in the case where we suppose one of the parameter to be known and we estimate the other one, as it simplifies substantially the expression of the estimators. On the one hand, if $a$ is known, the MLE for $b$ is given by
\begin{equation}\wc{b}_T=\frac{X_T-aT}{\int_0^T{X_t \, \mathrm{d}t}}
\end{equation}
and satisfies the following CLT
$$\sqrt{T}\left(\wc{b}_T-b\right)  \xrightarrow{\mathcal{L}} \mathcal{N}(0,4/\dE\left[X\right])$$ 
where $\dE\left[X\right]=-a/b$, see for instance \cite{Ov}. On the other hand, if $b$ is known, the MLE of $a$ is given by 
\begin{equation}
\wc{a}_T=\frac{\int_0^T{1/X_t \, \mathrm{d}X_t-bT}}{\int_0^T{1/X_t \, \mathrm{d}t}}
\end{equation}
and satisfies the following CLT
$$\sqrt{T}\left(\wc{a}_T-a\right)  \xrightarrow{\mathcal{L}} \mathcal{N}(0,4/\dE\left[X^{-1}\right])$$
with  $\dE\left[X^{-1}\right]=-b/(a-2)$.
Whereas, the weighted least squares estimators are respectively given by
\begin{equation*}\wh{b}_T=\frac{\int_0^T{\frac{X_t}{C_t} \, \mathrm{d}X_t-a \, \int_0^T{\frac{X_t}{C_t} \, \mathrm{d}t}}}{\int_0^T{\frac{X_t^2}{C_t} \, \mathrm{d}t}}
\hspace{1cm} \text{and} \hspace{1cm}\wh{a}_T=\frac{\int_0^T{\frac{1}{C_t} \, \mathrm{d}X_t-b \, \int_0^T{\frac{X_t}{C_t} \, \mathrm{d}t}}}{\int_0^T{\frac{1}{C_t} \, \mathrm{d}t}}.
\end{equation*}

\begin{prop}
Assume that $a>0$ and $b<0$. For $T$ going to infinity, $\wh{b}_T$ satisfies the following CLT:
\begin{equation}\label{P1}
\sqrt{T}\left(\wh{b}_T-b\right)  \xrightarrow{\mathcal{L}} \mathcal{N}\left(0,4 \, \dE\left[X^3/C^2\right]\left(\dE\left[X^2/C\right]\right)^{-2}\right).
\end{equation}
\end{prop}

\begin{proof} Replacing $\mathrm{d}X_t$ by its expression (\ref{heston}), we easily get that
$$\sqrt{T}\left(\wh{b}_T-b\right) = 2 \left(\frac{1}{T}\int_0^T{\frac{X_t^2}{C_t} \, \mathrm{d}t}\right)^{-1} \frac{n_T}{\sqrt{T}}$$
where $n_T$ is a martingale given by
\begin{equation*}\label{n}n_T=\int_0^T \frac{X_t\sqrt{X_t}}{C_t} \, \mathrm{d}B_t
\hspace{1cm} \text{and} \hspace{1cm}
\left\langle n\right\rangle_T= \int_0^T \frac{X_t^3}{C_t^2} \, \mathrm{d}t.
\end{equation*}
 Using the ergodicity of the process, we obtain for $T$ going to infinity
\begin{equation}\label{nc}
\frac{\left\langle n\right\rangle_T}{T} \xrightarrow{a.s.} \dE\left[X^3/C^2\right].
\end{equation}
Thus, by the CLT for martingales, we obtain the following convergence in distribution
\begin{equation}\label{loin} \frac{ n_T}{\sqrt{T}} \xrightarrow{\mathcal{L}} \mathcal{N}\left(0,\dE\left[X^3/C^2\right]\right).
\end{equation}
Consequently, \eqref{P1} follows from \eqref{loin}, Slutsky's lemma and the fact that, by the ergodicity of the process, $\frac{1}{T}\int_0^T{X_t^2/C_t \, \mathrm{d}t}$ converges a.s. to $\dE\left[X^2/C\right]$ for $T$ going to infinity.
\end{proof}

\begin{prop}
Assume that $a>0$ and $b<0$. For $T$ going to infinity, $\wh{a}_T$ satisfies the following CLT:
\begin{equation}\label{P2}
\sqrt{T}\left(\wh{a}_T-a\right)  \xrightarrow{\mathcal{L}} \mathcal{N}\left(0,4 \, \dE\left[X/C^2\right]\left(\dE\left[1/C\right]\right)^{-2}\right).
\end{equation}
\end{prop}

\begin{proof}
It works as in the previous proof. One can observe that $$\sqrt{T}\left(\wh{a}_T-a\right) = 2 \left(\frac{1}{T}\int_0^T{\frac{1}{C_t} \, \mathrm{d}t}\right)^{-1} \frac{m_T}{\sqrt{T}}$$
where $m_T$ is a martingale term given by
\begin{equation*}\label{m}m_T=\int_0^T \frac{\sqrt{X_t}}{C_t} \, \mathrm{d}B_t \hspace{1cm} \text{ and} \hspace{1cm} \left\langle m\right\rangle_T=\int_0^T \frac{X_t}{C_t^2} \, \mathrm{d}t .
\end{equation*} Thus, for $T$ going to infinity,
\begin{equation}\label{mc}
\frac{\left\langle m\right\rangle_T}{T} \xrightarrow{a.s.} \dE\left[X/C^2\right]
\end{equation}
which implies the following convergence in distribution 
\begin{equation}\label{loim} \frac{ m_T}{\sqrt{T}} \xrightarrow{\mathcal{L}} \mathcal{N}\left(0,\dE\left[X/C^2\right]\right).
\end{equation}
Finally, \eqref{loim} leads to \eqref{P2} thanks to the ergodicity of the process and Slutsky's lemma.
\end{proof}

\begin{prop}
Assume that $a>2$ is known and $b<0$. Then, the MLE of $b$ satisfies a CLT with a smaller asymptotic variance than the weighted least squares estimator.
\end{prop}

\begin{proof}
Using Cauchy-Schwarz Inequality, we notice that
$$\left(\dE\left[X^2/C\right]\right)^{2} = \left(\dE\left[\sqrt{X} \times X^{3/2}/C\right]\right)^{2} \leq \dE\left[X\right] \dE\left[X^3/C^2\right]$$
which immediately leads to the result.
\end{proof}

\begin{prop}
Assume that $a>2$ and $b<0$ is known. Then, the MLE of $a$ satisfies a CLT with a smaller asymptotic variance than the weighted least squares estimator.
\end{prop}

\begin{proof}
Using Cauchy-Schwarz Inequality, we notice that
$$\left(\dE\left[1/C\right]\right)^{2} = \left(\dE\left[X^{-1/2} \times X^{1/2}/C\right]\right)^{2} \leq \dE\left[1/X\right] \dE\left[X/C^2\right]$$
which immediately leads to the result.
\end{proof}

\begin{rem}
Thus, the weighted least squares estimator is less efficient than the MLE in the case where this later is easily manageable. This could seem to be contradictory to Remark 4.4 of \cite{Wei} which deals with the discrete-time counterpart of the process. In fact, they compare the weighted least squares with the conditional least squares estimator which does not coincide with the MLE.
\end{rem}

\begin{rem}
One could wonder how to choose which estimator to use, as the parameter $a$ is unknown. However, we suppose that we observe the whole trajectory of the process over the time interval $[0,T]$. Thus, if we are able to detect some local time at level zero, we know that $a<2$ and we should use the WLSE instead of the MLE.
\end{rem}
 
\section{Technical Lemmas}

In order to prove Theorem~\ref{T1}, we need to investigate the almost sure convergence of all the integrals involved in the definition of the estimators. Overbeck recalls in Lemma 3\textit{(i)} of \cite{Ov} that, for $T$ going to infinity, $X_T$ converges in distribution to $X$ with Gamma $\Gamma(a/2,-b/2)$ distribution, whose probability density function is given by
\begin{equation}\label{dens}
f(x)=\left(\Gamma(a/2)\right)^{-1} \left(-b/2\right)^{a/2} x^{a/2-1} e^{xb/2} \mathbf{1}_{x>0}.
\end{equation}
Thus, by Lemma 3\textit{(ii)} of \cite{Ov}, for $T$ going to infinity,
$$\frac{1}{T}\int_0^T g\left(X_t\right) \, \mathrm{d}t \xrightarrow{a.s.} \dE\left[g(X)\right] = \int_0^{+\infty} g(x) f(x) \, \mathrm{d}x.$$ for any function $g$ such that the right-hand side exists.

By an integration by part, we easily show the two following properties of the incomplete gamma function, which will be very useful in the following proof:
\begin{equation}\label{R1}
\Gamma(\alpha+1,x)=x^{\alpha} e^{-x} + \alpha \Gamma\left(\alpha,x\right)
\end{equation}
and
\begin{equation}\label{R2}
\Gamma(\alpha+2,x)=x^{\alpha} e^{-x} \left(x+\alpha+1\right) + \alpha \left(\alpha+1\right) \Gamma\left(\alpha,x\right).
\end{equation}
We are now able to prove the following lemmas. The three first points give us the almost sure limit of $T \Gamma_T^{-1}$ as $T$ goes to infinity, while the remaining deals with the increasing process of the right-hand side two-dimensional martingale of (\ref{Est}).
\begin{lem} With $\psi_c$ given by \eqref{defpsi}, we have that
\begin{enumerate}[label={(\roman*)},ref={(\roman*)}]
 \item \label{L1} $\dE \left[ 1/C\right]= \frac{\psi_c}{c}$.
 \item \label{L2}
$ \dE \left[ X/C\right]=1-\psi_c$.
 \item \label{L3} $\dE \left[ X^2/C\right]=c \left(\psi_c-1 \right) -\frac{a}{b}$.
 \item \label{L4}
$\dE\left[X/C^2\right]=\frac{a}{2c}\psi_c + \frac{b}{2} \left(1-\psi_c\right).$
\item \label{L5}
$\dE\left[X^2/C^2\right]=\frac{1}{2}\left(\left(a+2-bc\right)\left(1-\psi_c\right)-a\right)$.
\item \label{L6}
$\dE\left[X^3/C^2\right]=\frac{c}{2}\, \left(a+4-b\right) \psi_c -2c-\frac{bc^2}{2}-\frac{a}{b}.$
\end{enumerate}
\end{lem}

\begin{proof}
\textit{(i)} We have
\begin{equation}\label{E1}
 \dE \left[ 1/C\right]=\int_0^{+\infty} \frac{1}{x+c} f(x) \, \mathrm{d}x,
\end{equation}
where $f$ is given by \eqref{dens}.
  Formula 3.38(10) of \cite{GR} gives that
 \begin{equation*}
 \int_0^{+\infty}\frac{1}{x+c} x^{a/2-1} e^{xb/2} \, \mathrm{d}x = c^{a/2-1} e^{-bc/2} \Gamma(a/2) \Gamma(1-a/2,-bc/2),
 \end{equation*}
which leads to
 \begin{equation*}
\dE \left[ 1/C\right] =  \frac{1}{c} \left(-\frac{bc}{2}\right)^{a/2} e^{-bc/2} \, \Gamma\left(1-a/2,-bc/2\right)
 \end{equation*}
 and ensures the announced result.
 
\textit{(ii)} As in the previous proof, we have
\begin{equation}\label{E2}
\dE \left[ X/C\right]= \int_0^{+\infty} \frac{x}{x+c} f(x) \, \mathrm{d}x.
\end{equation}
By formula 3.38(10) of \cite{GR}, we know that
 \begin{equation}\label{E3}
 \int_0^{+\infty}\frac{1}{x+c} x^{a/2} e^{xb/2} \, \mathrm{d}x = c^{a/2} e^{-bc/2} \Gamma(a/2+1) \Gamma(-a/2,-bc/2).
 \end{equation}
 With formula (\ref{R1}), we easily obtain that 
 \begin{equation}\label{E6}
 \Gamma(-a/2,-bc/2)= \left(-\frac{2}{a}\right)\left(\Gamma\left(1-a/2,-bc/2\right)- \left(-\frac{bc}{2}\right)^{-a/2}e^{bc/2} \right).
\end{equation}  
Combining (\ref{E2}), (\ref{E3}), (\ref{E6})  and the fact that $\Gamma(a/2+1)=a/2 \times \Gamma(a/2)$, we deduce the announced result.

\textit{(iii)}
We have
\begin{equation*}
\dE\left[\frac{X^2}{C}\right] = \dE\left[ \frac{\left(X+c-c\right)^2}{X+c}\right] =  \dE\left[ X\right]-c+ \, c^2 \, \dE\left[ \frac{1}{C}\right],
\end{equation*}
and we conclude using \ref{L1} and the fact that $\dE\left[X\right]=-a/b$.

\textit{(iv)} By the very definition of $f$ given by \eqref{dens}, we have
\begin{equation}\label{E8}
\dE\left[X/C^2\right]=\int_0^{+\infty} \frac{x}{\left(x+c\right)^2} f(x) \, \mathrm{d}x =  \frac{\left(-b/2\right)^{a/2} }{\Gamma(a/2)} \int_0^{+\infty}\frac{x^{a/2}}{\left(x+c\right)^2} \,  e^{xb/2} \, \mathrm{d}x .
\end{equation}
Integrating the right-hand side of (\ref{E8}) by part, we obtain that
\begin{equation*}
\dE\left[X/C^2\right]= \frac{\left(-b/2\right)^{a/2} }{\Gamma(a/2)} \left[\frac{a}{2}\int_0^{+\infty}\frac{x^{a/2-1}}{x+c} \,  e^{xb/2} \, \mathrm{d}x  + \frac{b}{2}\int_0^{+\infty}\frac{x^{a/2}}{x+c} \,  e^{xb/2} \, \mathrm{d}x \right]. 
\end{equation*}
We have already computed both integrals in the proofs of respectively \ref{L1} and \ref{L2}, which leads to 
\begin{equation}\label{E9}
\dE\left[X/C^2\right]=\frac{a}{2} \, \dE\left[1/C\right]+ \frac{b}{2}  \,  \dE\left[X/C\right]= \frac{a}{2c} \, \psi_c + \frac{b}{2}  \, \left(1-\psi_c\right).
\end{equation}

\textit{(v)} Integrating by parts and using \ref{L3} and \ref{L4},
\begin{equation*}
\begin{array}{lcl}
\displaystyle \dE\left[X^2/C^2\right] &=& \displaystyle \frac{\left(-b/2\right)^{a/2} }{\Gamma(a/2)} \int_0^{+\infty}\frac{x^{a/2+1}}{\left(x+c\right)^2} \,  e^{xb/2} \, \mathrm{d}x \\
& = & \displaystyle \frac{\left(-b/2\right)^{a/2} }{\Gamma(a/2)} \left[ \frac{a+2}{2} \int_0^{+\infty}\frac{x^{a/2}}{x+c} \,  e^{xb/2} \, \mathrm{d}x + \frac{b}{2} \int_0^{+\infty}\frac{x^{a/2+1}}{x+c} \,  e^{xb/2} \, \mathrm{d}x \right]\\
&=&\displaystyle \left(\left(a+2\right) \, \dE\left[X/C\right] + b \, \dE\left[X^2/C\right]\right)/2 \\
&=&\displaystyle  \left(\left(a+2\right) \left(1-\psi_c\right) + b \left( c\left(\psi_c-1\right) -\frac{a}{b}\right)\right)/2 .
\end{array}
\end{equation*}
\textit{(vi)} Noticing that $X^3= X \left(X+c\right)^2 - 2c X^2 -c^2 X$, we obtain that
$$\dE\left[X^3/C^2\right]= \dE\left[X\right]-2c \,  \dE\left[X^2/C^2\right]- c^2 \, \dE\left[X/C^2\right]$$ and we conclude using \ref{L4} and \ref{L5}.

\end{proof}


\section{Proof of the strong Consistency}
We are now in the position to prove Theorem \ref{T1}. We first rewrite (\ref{Est}) using (\ref{heston}):

\begin{equation}\label{E14}
\wh{\theta}_T
= \theta +\begin{pmatrix}
\Gamma_T^{-1} & 0 \\
0 & \Gamma_T^{-1}
\end{pmatrix}\begin{pmatrix}
M_T\\
N_T
\end{pmatrix},
\end{equation}
where $M_T$ and $N_T$ are martingales respectively given by 
\begin{equation*}
M_T= \begin{pmatrix}
\displaystyle \int_0^T{\frac{2 \sqrt{X_t}}{C_t}  \, \mathrm{d}B_t}\\
\displaystyle \int_0^T{\frac{ 2 \sqrt{X_t}X_t}{C_t} \, \mathrm{d}B_t}
\end{pmatrix}
\hspace{1cm} \text{ and } \hspace{1cm}
N_T= \begin{pmatrix}
\displaystyle \int_0^T{\frac{2 \sqrt{X_t}}{C_t}  \, \mathrm{d}\wt{B}_t}\\
\displaystyle \int_0^T{\frac{ 2 \sqrt{X_t}X_t}{C_t} \, \mathrm{d}\wt{B}_t}
\end{pmatrix}
\end{equation*}
with $\mathrm{d}\wt{B}_t= \rho \,\mathrm{d}B_t + \sqrt{1-\rho^2} \, \mathrm{d}W_t$.
We denote by $\cM_T$ the martingale $\cM_T=(
M_T,N_T)$.
As $\left\langle \mathrm{d}B_t, \mathrm{d}\wt{B}_t \right\rangle = \rho \mathrm{d}t$, we easily obtain that the increasing process of $\cM_T$ is given by
\begin{equation}
\left\langle \cM\right\rangle_T = \begin{pmatrix}
\left\langle M \right\rangle_T & \rho \left\langle M \right\rangle_T \\
\rho \left\langle M \right\rangle_T  & \left\langle M \right\rangle_T 
\end{pmatrix}.
\end{equation}

\begin{proof}[Proof of Theorem \ref{T1}]
First of all, we have
$$\frac{1}{T^2} \det \Gamma_T = \frac{1}{T}\int_0^T{\frac{1}{C_t} \, \mathrm{d}t} \times \frac{1}{T} \int_0^T{\frac{X_t^2}{C_t} \, \mathrm{d}t}- \left(\frac{1}{T}\int_0^T{\frac{X_t}{C_t} \, \mathrm{d}t}\right)^2.$$
Thus, as the process is ergodic, we obtain for $T$ going to infinity,
\begin{equation}
 \frac{1}{T^2} \det \Gamma_T \xrightarrow{a.s.} \, \dE\left[1/C\right] \dE\left[X^2/C\right] - \left(\dE\left[X/C\right]\right)^2
 \end{equation}
and
\begin{equation}\label{E15}
T \Gamma_T^{-1} \xrightarrow{a.s.} A
\end{equation}
where $A$ is given by
\begin{equation}\label{E16}
A=\left(\dE \left[ C\right]\dE \left[ 1/C\right]-1\right)^{-1}\begin{pmatrix}
 \dE \left[ X^2/C\right] & -\dE \left[ X/C\right] \\
 -\dE \left[ X/C\right] & \dE \left[ 1/C\right]
\end{pmatrix}.
\end{equation}
A straightforward application of Lemmas 4.1 \ref{L1} to \ref{L3} gives that
\begin{equation*}
A= \frac{1}{\psi_c \left(1-\frac{a}{bc} \right) -1} \begin{pmatrix}
c \left(\psi_c-1 \right) -\frac{a}{b} & \psi_c-1 \\
\psi_c-1 & \frac{\psi_c}{c}\\
\end{pmatrix} .
\end{equation*}
Besides, the martingale $M_T$ satisfies for $T$ going to infinity
\begin{equation}\label{E17}
\frac{M_T}{T} \xrightarrow{a.s.} 0.
\end{equation}
As a matter of fact, by convergences (\ref{nc}) and (\ref{mc}), we know that a.s. $\left\langle n\right\rangle_T = \cO\left(T\right)$ and $\left\langle m\right\rangle_T= \cO\left(T\right)$. It ensures that for $T$ going to infinity, 
\begin{equation*}
\frac{n_T}{T} \xrightarrow{a.s.} 0 \hspace{1cm} \text{ and } \hspace{1cm} \frac{m_T}{T} \xrightarrow{a.s.} 0.
\end{equation*}
As $N_T$ and $M_T$ share the same increasing process, this result remains true by replacing $M_T$ by $N_T$.
Finally, the almost sure convergence \eqref{cvps_hest} follows from (\ref{E14}), (\ref{E15}) and (\ref{E17}).

\end{proof}

\section{Proof of the asymptotic normality}

\begin{proof}[Proof of Theorem \ref{T2}]
First of all, we deduce from (\ref{E14}) that
\begin{equation}\label{six}
 \sqrt{T} \left(
\wh{\theta}_T
-\theta\right) =\begin{pmatrix}
T\Gamma_T^{-1} & 0 \\
0 & T\Gamma_T^{-1}
\end{pmatrix}\begin{pmatrix}
M_T/ \sqrt{T}\\
N_T/\sqrt{T}
\end{pmatrix},
\end{equation}
We already saw that $T \Gamma_T^{-1}$ converges almost surely as $T$ goes to infinity and its limit $A$ is given by (\ref{E16}).
We now have to establish the asymptotic normality of $\frac{M_T}{\sqrt{T}}$. The increasing process of $M_T$ is given by 
\begin{equation*}
\left\langle M \right\rangle_T = 4 \begin{pmatrix}
\left\langle m \right\rangle_T & \displaystyle \int_0^T \frac{X_t^2}{C_t^2} \, \mathrm{d}t \\
 \displaystyle \int_0^T \frac{X_t^2}{C_t^2} \, \mathrm{d}t & \left\langle n \right\rangle_T
\end{pmatrix}
\end{equation*}
where $\left\langle m \right\rangle_T$ and $\left\langle n \right\rangle_T$ are respectively given by \eqref{mc} and \eqref{nc}.
 Thus, by the ergodicity of the process, we obtain that
\begin{equation*}
\frac{\left\langle M \right\rangle_T}{T} \xrightarrow{a.s.} 4 L 
\hspace{1cm} \text{where} \hspace{1cm}
L=\begin{pmatrix}
\dE \left[ X/C^2\right] & \dE \left[ X^2/C^2\right] \\
\dE \left[ X^2/C^2\right] & \dE \left[ X^3/C^2\right]
\end{pmatrix}.
\end{equation*}
As a straightforward consequence of Lemmas 4.1 \ref{L4} to \ref{L6}, we obtain that 
 \begin{equation*}L = \frac{1}{2} \begin{pmatrix}
\displaystyle \frac{a}{c} \, \psi_c + b \left(1-\psi_c\right)& \displaystyle \left(a+2-bc\right) \, \left(1-\psi_c\right) - a\\
\displaystyle \left(a+2-bc\right) \, \left(1-\psi_c\right) - a & \displaystyle \psi_c c\, \left(a+4-b\right) -4c-bc^2-\frac{2a}{b}
\end{pmatrix}.
\end{equation*}
 We easily obtain the following a.s. convergence 
\begin{equation*}
\frac{\left\langle \cM \right\rangle_T}{T} \xrightarrow{a.s.} 4 \cL 
\end{equation*}
 where $\cL$ is a block matrix given by
 \begin{equation*}
 \cL= \begin{pmatrix}
 L & \rho L \\
 \rho L & L 
 \end{pmatrix}.
 \end{equation*}and we deduce from the CLT for martingales that
 \begin{equation}\label{loiM}
 \frac{\cM_T}{\sqrt{T}} \xrightarrow{\cL} \cN\left(0,4 \cL\right),
 \end{equation}
Finally, the asymptotic normality \eqref{TCL_hest} follows from \eqref{six} and \eqref{loiM} together with Slutsky's Lemma.
\end{proof}


\section{Numerical simulations}
The efficient discretization of the CIR process is a challenging question, see for example \cite{And} and \cite{Alf}.
We choose to implement the QE-algorithm based on quadratic-exponential approximations proposed in \cite{And}. Andersen introduced this algorithm to deal with the case $a<2$, for which common discretization schemes are not accurate. 
\subsection{Asymptotic behavior for $c=1$}
The two following figures illustrate our main results (strong consistency and asymptotic normality) in the case $a=1$ and $b=-2$, with the weighting parameter $c=1$. The red curves in the second figure displays the standard normal distribution. 
\begin{figure}[!ht]
\includegraphics[width=10cm, height=7cm]{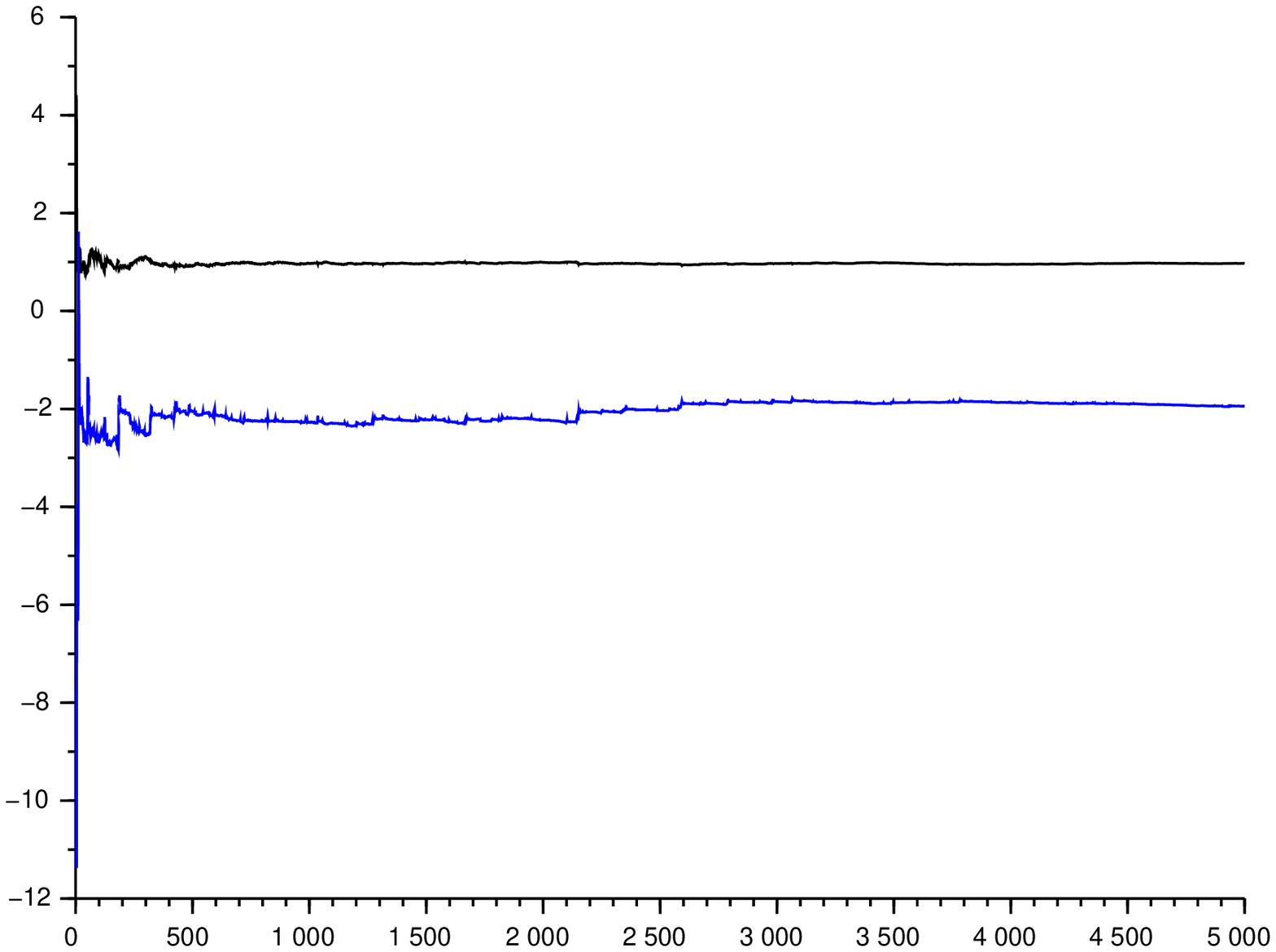}
\caption{strong consistency: $(\wh{a}_T)$ in black and $(\wh{b}_T)$ in blue.}
\label{ft1}
\end{figure}

\begin{figure}[!ht]
\includegraphics[width=10cm, height=7cm]{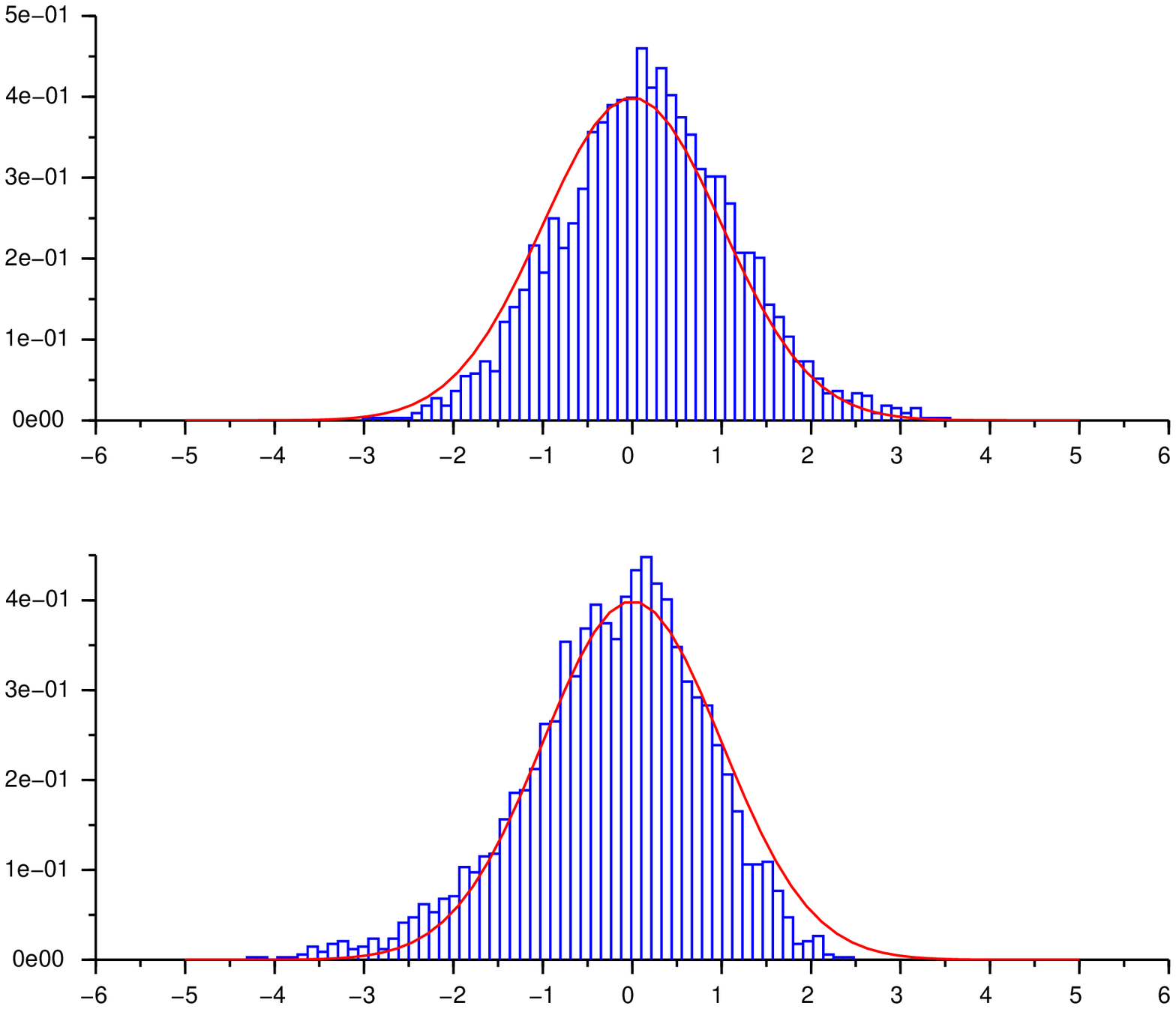}
\caption{Histograms of $3000$ outcomes of $\sqrt{T/4\sigma_{11}} \left(\wh{a}_T-a\right)$ and $\sqrt{T/4\sigma_{22}} (\wh{b}_T-b)$ at time $T=70$ . }
\label{ft2}
\end{figure}

\FloatBarrier

\subsection{Choice of the constant $c$}
We have chosen to introduce a constant $c$ in our weighting, instead of only considering the case $c=1$ (as done in the discrete-time case in \cite{Wei}) with the aim of lowering the variance of the estimators. However, this raises the question of the optimal choice of the constant $c$, which depends on the values of parameters $a$ and $b$. We set $a=1$ and $b=-4$ and simulate $500$ trajectories of the process over the time interval $[0,50]$. We compute the empirical variance of the estimators given by each trajectory for $c$ varying between $10^{-10}$ and $1$. It appears that one should choose a small value of $c$.  The value should not be to small to avoid the growth illustrated by the second figure, which might however be a consequence of the discretized version of the CIR process we used. For $\wh{a}_T$, there is a significant difference (factor $5$) between the empirical variances obtained with $c=0.01$ and $c=1$. However, for $\wh{b}_T$ both empirical variances do not significantly differ.
\begin{figure}[!h]
\begin{minipage}[c]{.46\linewidth}
      \includegraphics[width=7cm, height=7cm]{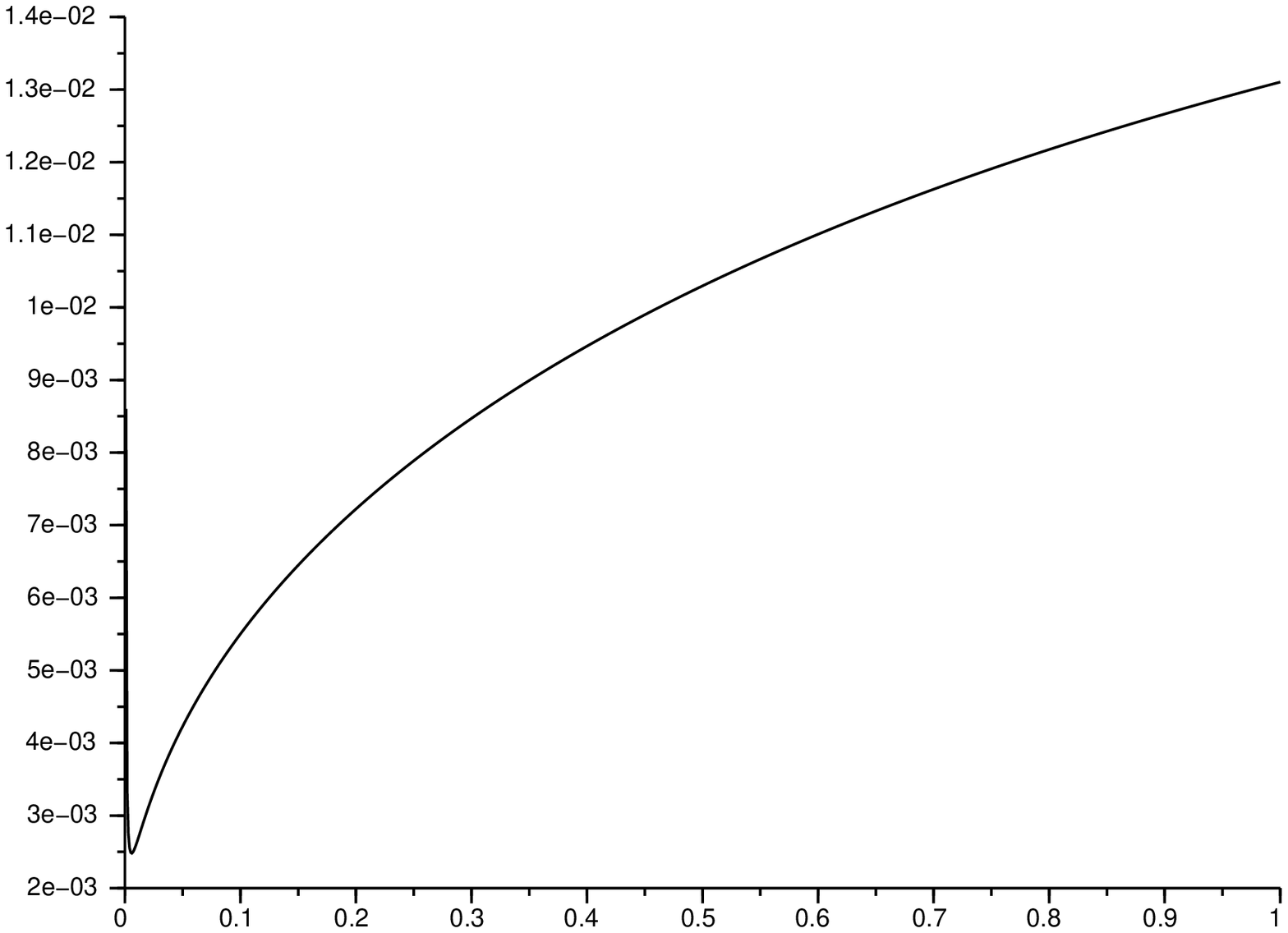}
      \caption{Variance of $\wh{a}_{50}$}
   \end{minipage} \hfill
   \begin{minipage}[c]{.46\linewidth}
      \includegraphics[width=7cm, height=7cm]{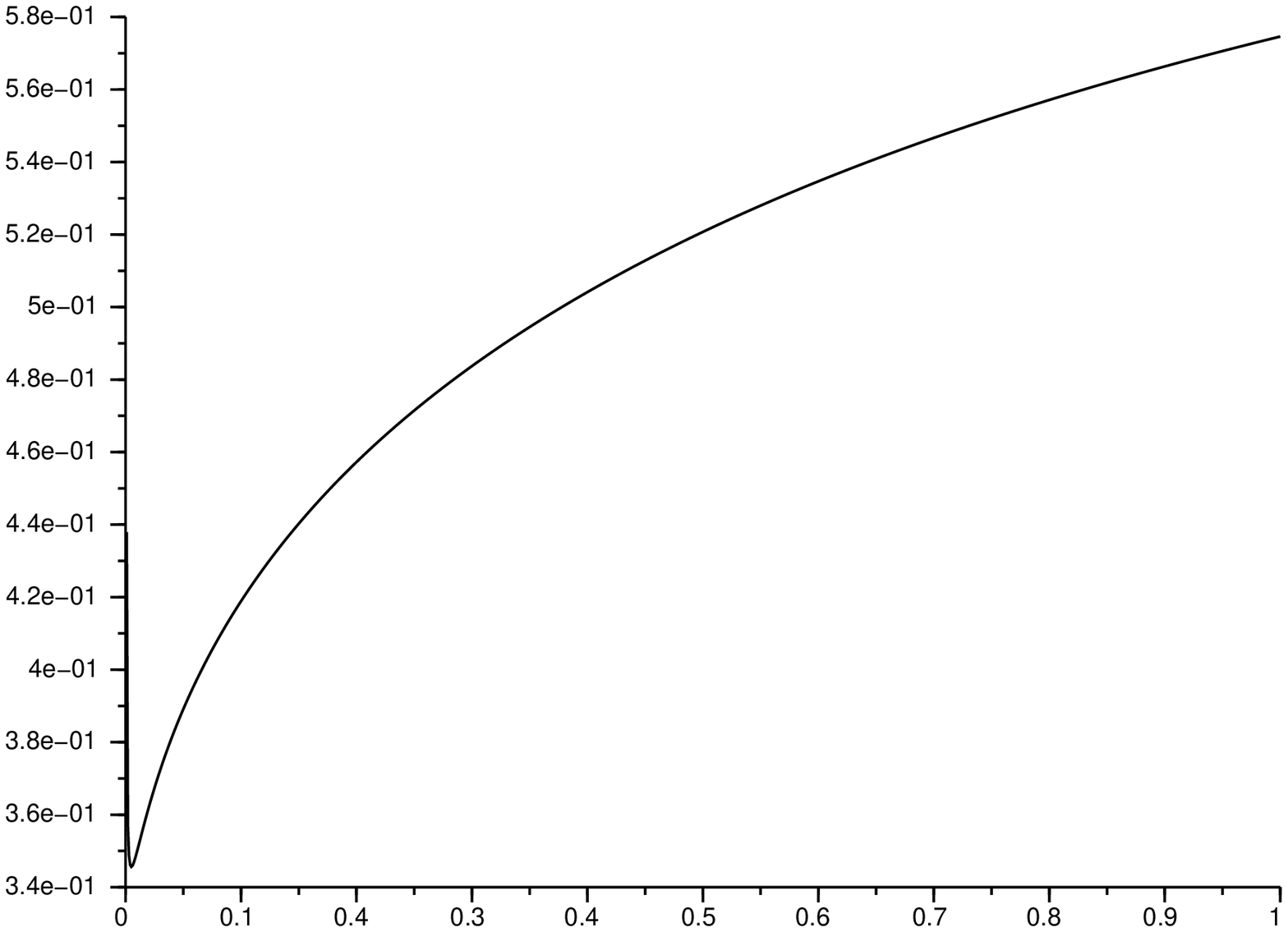}
      \caption{Variance of $\wh{b}_{50}$}
   \end{minipage}
   \unnumberedcaption{Empirical variances of the estimators}
\end{figure}

\begin{figure}[!h]
\begin{minipage}[c]{.46\linewidth}
      \includegraphics[width=7cm, height=7cm]{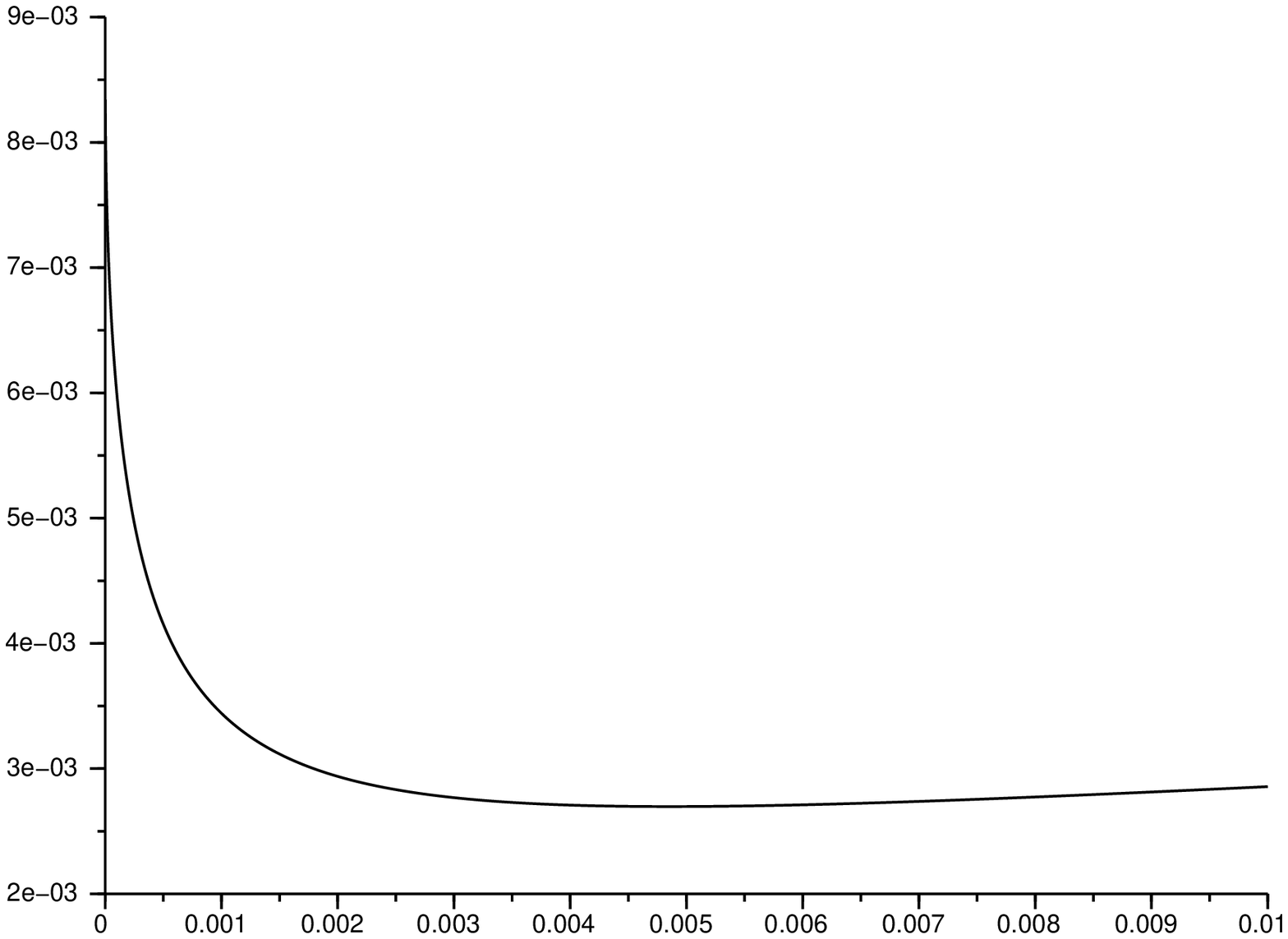}
      \caption{Variance of $\wh{a}_{50}$}
   \end{minipage} \hfill
   \begin{minipage}[c]{.46\linewidth}
      \includegraphics[width=7cm, height=7cm]{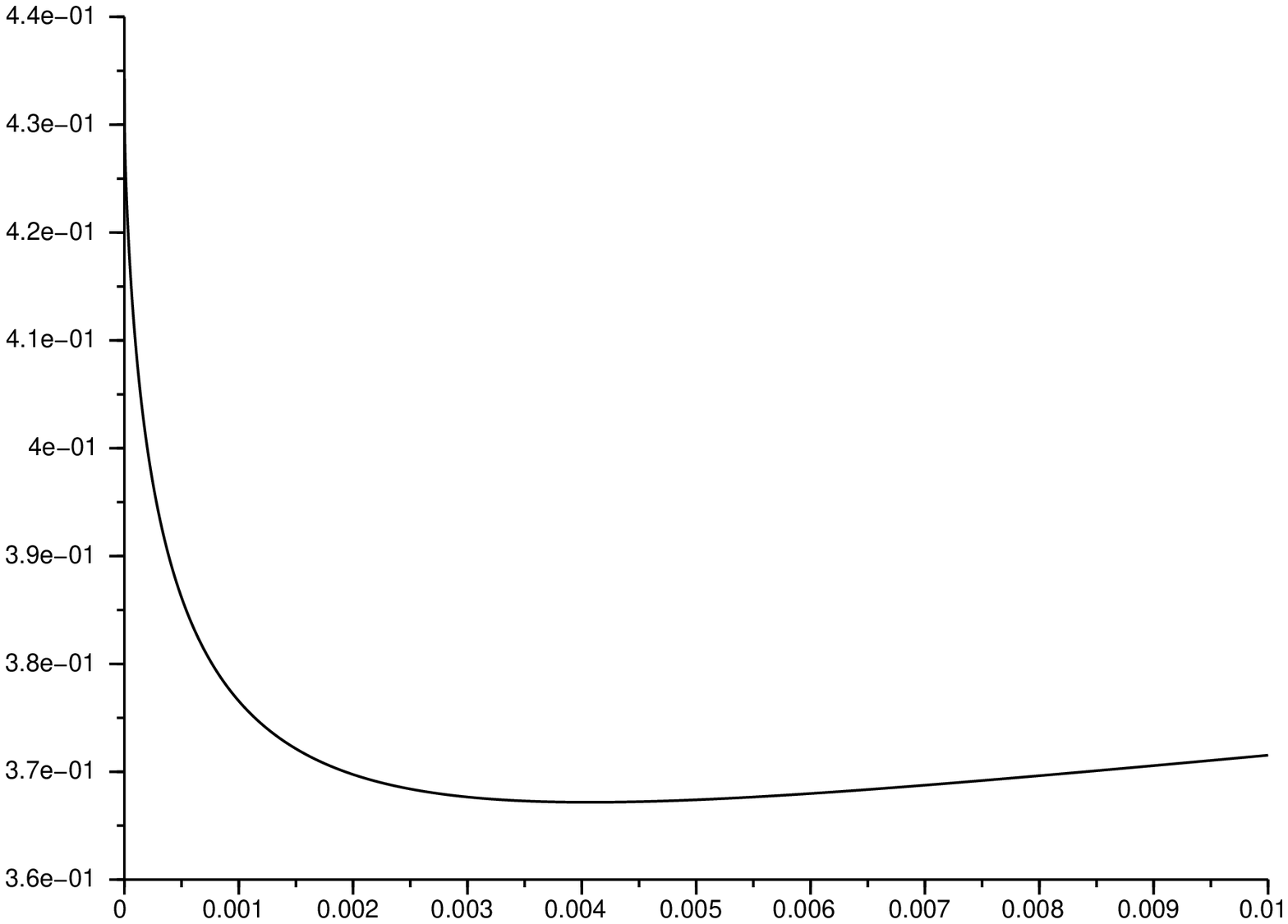}
      \caption{Variance of $\wh{b}_{50}$}
   \end{minipage}
   \unnumberedcaption{Empirical variances for very small values of $c$}
\end{figure}

\FloatBarrier

\bibliographystyle{acm}
\bibliography{biblio} 

\end{document}